\documentclass[11pt]{article}

\usepackage{amsmath}
\usepackage{amsthm}
\usepackage{amssymb}
\usepackage{amsfonts}
\newcommand{\field}[1]{\mathbb{#1}}
\newtheorem{thm}{Theorem}[section]

\newtheorem{remk}[thm]{Remark}
\allowdisplaybreaks
\usepackage{graphicx}

\newtheorem{theorem}{Theorem}[section]
\newtheorem{corollary}[thm]{Corollary}

\newtheorem{lemma}[thm]{Lemma}
\newtheorem{proposition}[thm]{Proposition}

\usepackage{fullpage}

\begin{document}
\title{On a $L^\infty$ functional derivative estimate relating to the Cauchy problem for scalar semi-linear parabolic partial differential equations with general continuous nonlinearity}
\author{John Christopher Meyer and David John Needham}
\maketitle


\begin{abstract}
In this paper, we consider a $L^\infty$ functional derivative estimate for the first spatial derivative of bounded classical solutions $u:\field{R}\times [0,T]\to\field{R}$ to the Cauchy problem for scalar semi-linear parabolic partial differential equations with a continuous nonlinearity $f:\field{R}\to\field{R}$ and initial data $u_0:\field{R}\to\field{R}$, of the form, 
\[ \sup_{x\in\field{R}}|u_x (x , t)| \leq \mathcal{F}_t (f,u_0,u) \ \ \ \forall t\in [0,T] . \] 
Here $\mathcal{F}_t:\mathcal{A}_t\to\field{R}$ is a functional as defined in \textsection 1. We establish that the functional derivative estimate is non-trivially sharp, by constructing a sequence $(f_n,0,u^{(n)})$, where for each $n\in\field{N}$, $u^{(n)}:\field{R}\times [0,T]\to\field{R}$ is a solution to the Cauchy problem with zero initial data and nonlinearity $f_n:\field{R}\to\field{R}$, and for which $\sup_{x\in\field{R}} |u_x^{(n)}(x,T)| \geq \alpha >0$, with
\[ \lim_{n\to\infty} \left( \inf_{t\in [0,T]}  \left(  \sup_{x\in\field{R}}|u_x^{(n)}(\cdot , t)| - \mathcal{F}_t (f_n , 0 , u^{(n)}) \right) \right) = 0 . \]
\end{abstract}

\section{Introduction}
To begin, for each $\lambda >0$, we introduce the sets
\[ \partial D = \field{R}\times \{ 0 \},\ \ \ D_\lambda = \field{R}\times (0,\lambda ],  \]
together with the Banach spaces $(B_A^\lambda , ||\cdot ||_A^\lambda)$ and $(B_B,||\cdot ||_B)$ where
\[ B_A^\lambda = \{ u:\bar{D}_\lambda\to\field{R}\ :\ u\text{ is continuous and bounded.}\} \]
\[ B_B = \{ v:\field{R}\to\field{R}\ :\ v\text{ is continuous and bounded.}\}  \] 
and
\[ || u ||_A^\lambda = \sup_{(x,t)\in \bar{D}_\lambda} |u(x,t)| \ \ \ \forall u\in B_A^\lambda \]
\[ || v ||_B = \sup_{x\in \field{R}} |v(x)| \ \ \ \forall v\in B_B .\]
Additionally, for $\lambda >0$, we define the set 
\[ \mathcal{A}_\lambda = \{ (f,v,u)\ :\ f\in C(\field{R}),\ v\in \text{BPC}^1(\field{R}),\ u\in B_A^\lambda \} ,\]
where $\text{BPC}^1(\field{R})$ is the set of bounded, continuous functions with piecewise continuous bounded derivative. Finally, for $\lambda >0$, we introduce the functional $\mathcal{F}_\lambda : \mathcal{A}_\lambda \to [0,\infty )$ given by
\begin{equation} \label{func} \mathcal{F}_\lambda (f,v,u) = ||v'||_B + \frac{1}{\sqrt{\pi}}\int_0^\lambda \frac{||f(u(\cdot , \tau ))||_B}{(\lambda - \tau )^{1/2}}d\tau \ \ \ \forall (f,v,u)\in\mathcal{A}_\lambda . \end{equation}
We now have the following elementary results, 

\begin{lemma} \label{lem1.1}
For fixed $\lambda >0$, let $(f,v,u) \in \mathcal{A}_\lambda$, and let $G:[0,\lambda ]\to\field{R}$ be such that 
\[ G(\tau ) = \begin{cases} \frac{ ||f(u(\cdot , \tau ))||_B}{(\lambda-\tau )^{1/2}} &, \tau \in [0, \lambda ) \\ 0 &, \tau = \lambda . \end{cases} \]
Then $G\in L^1([0,\lambda ])$, and 
\[ 0 \leq \int_0^\lambda G(\tau ) d\tau = \lim_{m\to\infty} \int_0^{\lambda - \frac{1}{m}} G(\tau ) d\tau \leq 2 \sqrt{\lambda} ||f(u)||_A^\lambda .\]
\end{lemma}

\begin{proof}
For fixed $\lambda >0$, let $(f,v,u)\in \mathcal{A}_\lambda$. Since $u\in B_A^\lambda$ and $f\in C(\field{R})$, then $f(u)\in B_A^\lambda$ and $f(u(\cdot , \tau ))\in B_B$ for each $\tau \in [0,\lambda ]$. Thus, $G: [0,\lambda ]\to\field{R}$ is well-defined. Next, set $m\in\field{N}$ with $m\geq [\tfrac{1}{\lambda}] + 1:=m_\lambda$ and introduce the sequence of functions $\{ g_n^m:[0,\lambda ] \to\field{R} \}_{n\in\field{N}}$ given by,
\begin{equation} \label{ep1b} g_n^m(\tau ) = \begin{cases}\frac{\sup_{x\in [-n,n]}|f( u (x,\tau ))|}{(\lambda - \tau )^{1/2}}  &,\tau \in \left[ 0, \lambda -\tfrac{1}{m}\right) \\ 0 &, \tau \in \left[ \lambda - \tfrac{1}{m} , \lambda \right] .\end{cases} \end{equation}
Since $f(u)\in B_A^\lambda$, then for each $n\in\field{N}$, it follows that $g_n^m:[0,\lambda ] \to\field{R}$ is piecewise continuous, and so $g_n^m \in L^1 ([0,\lambda ])$. In addition, for each $n\in\field{N}$, 
\begin{equation} \label{ep1f} 0 \leq g_n^m(\tau ) \leq g_{n+1}^m (\tau ) \ \ \ \forall \tau \in  [0,\lambda ], \end{equation}
and, for each $n\in\field{N}$,
\begin{equation} \label{ep1c} \int_0^\lambda g_n^m (\tau ) d\tau \leq 2 \sqrt{\lambda} ||f(u)||_A^\lambda . \end{equation}
Now, introduce $h^m:[0,\lambda ] \to\field{R}$ such that,
\begin{equation} \label{ep1d} h^m(\tau ) = \begin{cases} \frac{||f(u(\cdot , \tau ))||_B}{(\lambda - \tau )^{1/2}} &, \tau \in \left[ 0,\lambda - \tfrac{1}{m}\right) \\ 0 &, \tau \in \left[ \lambda - \tfrac{1}{m} , \lambda \right] .\end{cases} \end{equation}
It follows from \eqref{ep1b} and \eqref{ep1d} that for each $\tau \in [0,\lambda ]$, 
\begin{equation} \label{ep1e} g_n^m(\tau ) \to h^m (\tau ) . \end{equation}
An application of the monotone convergence theorem \cite[pp. 318]{WR1}, via \eqref{ep1f}, \eqref{ep1c} and \eqref{ep1e}, then establishes that $h^m\in L^1([0,\lambda ])$, for each $m \geq m_\lambda$. It follows from \eqref{ep1d}, that for each $m\in\field{N}$ ($m\geq m_\lambda$),
\begin{equation} \label{ep1g} 0 \leq h^m (\tau ) \leq h^{m+1} (\tau ) \ \ \ \forall \tau \in [0,\lambda ] \end{equation}
and,
\begin{equation} \label{ep1h} 0 \leq \int_0^\lambda h^m (\tau ) d\tau \leq 2 \sqrt{\lambda} ||f(u)||_A^\lambda , \end{equation}
whilst, for each $\tau \in [0,\lambda )$, 
\begin{equation} \label{ep1i} h^m(\tau ) \to G(\tau ) \ \ \ \text{ as }m\to\infty . \end{equation}
Finally, an application of the monotone convergence theorem, via \eqref{ep1g}, \eqref{ep1h} and \eqref{ep1i}, then establishes that $G\in L^1([0,\lambda ])$, and, moreover, that, 
\[ 0 \leq \int_0^\lambda G(\tau ) d\tau = \lim_{m\to\infty } \int_0^\lambda h^m(\tau ) d\tau = \lim_{m\to\infty} \int_0^{\lambda - \frac{1}{m}} G(\tau ) d\tau \leq 2 \sqrt{\lambda} ||f(u)||_A^\lambda , \]
as required. 
\end{proof}
\noindent We now have,

\begin{proposition} \label{DAVEARGHHH1}
For each $\lambda >0$, the functional $\mathcal{F}_\lambda : \mathcal{A}_\lambda \to [0,\infty )$ is well-defined, and, for each $(f,v,u)\in \mathcal{A}_\lambda$, then, 
\begin{equation} \label{ebound} \mathcal{F}_\lambda (f,v,u) \leq ||v'||_B + \frac{2\lambda^{1/2}}{\sqrt{\pi}}||f(u)||_A^\lambda  . \end{equation}
\end{proposition}

\begin{proof}
Let $(f,v,u)\in \mathcal{A}_\lambda$. Then, via Lemma \ref{lem1.1}, $G\in L^1([0,\lambda ])$, and, since $v\in \text{BPC}^1(\field{R})$, then $||v'||_B$ exists. Now, via \eqref{func}, 
\begin{equation} \label{p2a} \mathcal{F}_\lambda (f,v,u) = ||v'||_B + \frac{1}{\sqrt{\pi}}\int_0^\lambda G(\tau ) d\tau ,\end{equation}
and so $\mathcal{F}_\lambda : \mathcal{A}_\lambda \to [0,\infty )$ is well-defined. In addition, the inequality follows directly from \eqref{p2a} and Lemma \ref{lem1.1}. 
\end{proof}

\noindent In the remainder of the paper, we consider the sharpness of a derivative estimate for solutions {\mbox{$u:\bar{D}_T\to\field{R}$}} to the semi-linear parabolic initial value problem ($T>0$) given by,
\begin{equation} \label{1} u_t - u_{xx} = f(u) \ \ \  \text{ on } D_T, \end{equation}
\begin{equation} \label{2} u = u_0 \ \ \ \text{ on }\partial D , \end{equation} 
with nonlinearity $f\in C(\field{R})$ and initial data $u_0\in \text{BPC}^1(\field{R})$. We consider bounded solutions to the initial value problem \eqref{1}-\eqref{2} (henceforth referred to as [IVP]), which are classical, in the sense that 
\[ u\in C(\bar{D}_T)\cap C^{2,1}(D_T)\cap L^\infty (\bar{D}_T).\]
Related to [IVP], we introduce the set $\mathcal{I}_T \subset \mathcal{A}_T$ such that 
\begin{align}
\nonumber  \mathcal{I}_T = & \{ (f,u_0,u): (f,u_0,u)\in\mathcal{A}_T \text{ and }  \\
\label{2'} & \ \   u:\bar{D}_T\to\field{R} \text{ is a solution to [IVP] with }f \text{ and } u_0.  \}  
\end{align}
We observe (for any $T>0$) that $\mathcal{I}_T$ is non-empty (take $(f,v,u)\in \mathcal{A}_T$ with each being the zero function). We also observe, via \cite{JCMDJN4}, that for any $T>0$, when $f=f_p:\field{R}\to\field{R}$ (for any $0<p<1$) is given by
\begin{equation} \label{fp} f_p(u) = u|u|^{p-1} \ \ \ \forall u\in\field{R} \end{equation} 
and $u_0:\field{R}\to\field{R}$ is given by $u_0(x)=0$ for all $x \in\field{R}$, it has been established in \cite{JCMDJN4} that there exists non-trivial $u=u_p:\bar{D}_T\to\field{R}$ such that 
\[ (f_p,0,u_p)\in \mathcal{I}_T . \]
We will examine [IVP] with $f=f_p$ in detail, in \textsection 2 and \textsection 4. Now, we consider a Schauder-type derivative estimate for [IVP], which is a straightforward extension of those given in \cite[Lemma 5.12]{JMDN1} and \cite[Lemma 3.9]{JMDN3}. 

\begin{proposition}[Derivative Estimate] \label{lem233}
Let $(f,u_0,u)\in \mathcal{I}_T$. Then, for each $0<t\leq T$, it follows that $(f,u_0,u|_{\bar{D}_t})\in\mathcal{I}_t$ and   
\[ || u_x(\cdot , t)||_B \leq \mathcal{F}_t(f,u_0,u|_{\bar{D}_t}).\]
\end{proposition}

\begin{proof}
Since $u:\bar{D}_T\to\field{R}$ is a solution to [IVP] with $f$ and $u_0$, it follows by definition that $u|_{\bar{D}_t}$ is a solution to [IVP] with $f$ and $u_0$ on $\bar{D}_t$ then for any $0<t\leq T$, and hence, $(f,u_0,u|_{\bar{D}_t})\in \mathcal{I}_t$. For convenience we drop the restriction notation from here onward, (with $(f,u_0,u)\in\mathcal{A}_T$, then $(f,u_0,u)\in\mathcal{A}_t$ for each $0<t\leq T$). Now, let $(f,u_0,u)\in\mathcal{I}_T$. Then, since $f(u)\in B_A^T$, it follows via \cite[Theorem 4.9]{JMDN1} that,
\begin{align}
\nonumber u(x,t) = & \frac{1}{\sqrt{\pi}}\int_{-\infty}^{\infty} u_0(x+2\sqrt{t}w ) e^{-w^2}dw \\
\label{p3a} & + \frac{1}{\sqrt{\pi}}\int_{0}^{t} \int_{-\infty}^{\infty} f(u(x+2\sqrt{t-\tau} w ,\tau)) e^{-{w^2}} dw d\tau  \ \ \ \forall (x,t)\in D_T. 
\end{align}
Again, since $f(u)\in B_A^T$ and $u_0\in \text{BPC}^1(\field{R})$, we observe, via \cite[Lemma 5.9]{JMDN1} that both terms on the right hand side of \eqref{p3a} have continuous partial derivatives with respect to $x$ on $D_T$, which are given by,
\begin{equation} \label{psb} \left( \frac{1}{\sqrt{\pi}}\int_{-\infty}^{\infty} u_0(x+2\sqrt{t}w ) e^{-w^2}dw \right)_x = \frac{1}{\sqrt{\pi}}\int_{-\infty}^{\infty} u_0'(x+2\sqrt{t}w ) e^{-w^2}dw \ \ \ \forall (x,t)\in D_T , \end{equation}
\begin{align}
\nonumber & \left( \frac{1}{\sqrt{\pi}}\int_{0}^{t} \int_{-\infty}^{\infty} f(u(x+2\sqrt{t-\tau} w ,\tau)) e^{-{w^2}} dw d\tau \right)_x \\
\label{psc} & \hspace{10mm} = \frac{1}{\sqrt{\pi}}\lim_{\epsilon\to 0} \int_0^{t-\epsilon} \hspace{-2mm} \int_{-\infty}^\infty \frac{f(u(x+2\sqrt{t-\tau}w,\tau ))}{(t-\tau )^{1/2}} w e^{-w^2} dw d\tau \ \ \ \forall (x,t)\in D_T, 
\end{align}
and so,
\begin{align}
\nonumber u_x(x,t) =& \frac{1}{\sqrt{\pi}}\int_{-\infty}^{\infty} u_0'(x+2\sqrt{t}w ) e^{-w^2}dw \\
\label{psd} &\ + \frac{1}{\sqrt{\pi}}\lim_{\epsilon\to 0} \int_0^{t-\epsilon} \hspace{-2mm} \int_{-\infty}^\infty \frac{f(u(x+2\sqrt{t-\tau}w,\tau ))}{(t-\tau )^{1/2}} w e^{-w^2} dw d\tau \ \ \ \forall (x,t) \in D_T . 
\end{align}
Therefore, 
\begin{align}
\nonumber |u_x(x,t)|& \leq \  ||u_0'||_B  \\
\label{psf} & + \frac{1}{\sqrt{\pi}}\lim_{\epsilon\to 0} \int_0^{t-\epsilon} \hspace{-2mm} \int_{-\infty}^\infty \left|\frac{f(u(x+2\sqrt{t-\tau}w,\tau ))}{(t-\tau )^{1/2}} w e^{-w^2}\right| dw d\tau \ \ \forall (x,t)\in D_T . 
\end{align}
It follows from Lemma \ref{lem1.1} that
\begin{align}
\nonumber \lim_{\epsilon\to 0} & \int_0^{t-\epsilon} \hspace{-2mm} \int_{-\infty}^\infty \left| \frac{f(u(x+2\sqrt{t-\tau}w,\tau ))}{(t-\tau )^{1/2}} w e^{-w^2}\right| dw d\tau  \\ 
\nonumber & \leq \lim_{\epsilon\to 0} \int_0^{t-\epsilon}  \hspace{-2mm} \int_{-\infty}^\infty \frac{||f(u(\cdot ,\tau ))||_B}{(t-\tau )^{1/2}} |w| e^{-w^2} dw d\tau \\
\label{pse} & = \int_0^{t} \frac{||f(u(\cdot ,\tau ))||_B}{(t-\tau )^{1/2}} d\tau \ \ \ \forall (x,t)\in D_T .     
\end{align}
Therefore, from \eqref{psf} and \eqref{pse}, we have,
\begin{equation} \label{psg} |u_x(x,t)| \leq ||u_0'||_B + \frac{1}{\sqrt{\pi}} \int_0^t \frac{||f(u(\cdot , \tau ))||_B}{(t-\tau )^{1/2}} d\tau = \mathcal{F}_t(f,u_0,u) \ \ \ \forall (x,t)\in D_T . \end{equation}
Since the right hand side of \eqref{psg} is independent of $x\in\field{R}$ the result follows.
\end{proof}
\noindent A classical Schauder-type derivative estimate can now be obtained as follows,

\begin{corollary}[Schauder-type Estimate] \label{ec3}
Let $(f,u_0,u)\in \mathcal{I}_T$. Then,    
\[ || u_x(\cdot , t)||_B \leq ||u_0'||_B + \frac{2t^{1/2}}{\sqrt{\pi}} ||f( u)||_A^t \ \ \ \forall t\in (0,T] .\]
\end{corollary}

\begin{proof}
This follows directly from Proposition \ref{lem233} and Proposition \ref{DAVEARGHHH1}.
\end{proof}
\noindent For each $(f,u_0,u)\in \mathcal{I}_T$, we now have, via Proposition \ref{lem233} and Proposition \ref{DAVEARGHHH1}, that
\begin{equation} \label{ee2} -\left( ||u_0'||_B + \frac{2T^{1/2}}{\sqrt{\pi}} ||f( u )||_A^T\right) \leq ||u_x(\cdot , t) ||_B - \mathcal{F}_t(f,u_0,u) \leq 0 \ \ \ \forall t\in (0,T]. \end{equation}
Therefore, given $(f,u_0,u)\in\mathcal{I}_T$, then $(||u_x(\cdot , t) ||_B - \mathcal{F}_t(f,u_0,u)) $ is bounded uniformly above and below for $t\in [0,T]$. Moreover, via \eqref{ee2}, it follows that for any $(f,u_0,u)\in\mathcal{I}_T$,
\begin{equation} \label{ee3} -||u_0'||_B - \frac{2T^{1/2}}{\sqrt{\pi}} ||f( u )||_A^T \leq \inf_{t\in (0,T]} \left( ||u_x(\cdot , t) ||_B - \mathcal{F}_t (f,u_0,u) \right) \leq 0. \end{equation}
Now, motivated by Proposition \ref{lem233} and \eqref{ee3}, we refer to the derivative estimate in Proposition \ref{lem233} as \textit{sharp} on $\bar{D}_T$, if 
\[ \sup_{(f,u_0,u)\in\mathcal{I}_T} \left( \inf_{t\in (0,T]} \left( ||u_x(\cdot , t) ||_B - \mathcal{F}_t (f,u_0,u) \right) \right) = 0. \]
However, we observe that this definition is not immediately satisfactory due to the following example. Take, $f^*\in C(\field{R})$, $u_0^*\in \text{BPC}^1(\field{R})$ and $u^*\in B_A^T$ to be
\begin{equation} \label{eex1} f^*(u) = 0\ \ \ \forall u\in\field{R}, \end{equation} 
\begin{equation} \label{eex2} u_0^*(x) = 0\ \ \ \forall x\in\field{R}, \end{equation} 
\begin{equation} \label{eex3} u^*(x,t) = 0\ \ \ \forall (x,t)\in \bar{D}_T. \end{equation}
Then, trivially, $(f^*,u_0^*,u^*)\in\mathcal{I}_T$, with, 
\begin{equation} \label{eex4} ||u_x^*(\cdot , t)||_B = 0 \ \ \ \forall t\in (0,T] , \end{equation}
\begin{equation} \label{eex5} \mathcal{F}_t(f^*,u_0^*,u^*) = 0 \ \ \ \forall t\in (0,T]. \end{equation}
Thus, it follows from \eqref{eex4} and \eqref{eex5} that
\begin{equation} \label{eex6} \inf_{t\in (0,T]} \left( ||u_x^*(\cdot , t)||_B - \mathcal{F}_t(f^*,u_0^*,u^*) \right) = 0. \end{equation}
Finally, it follows from \eqref{eex6} and \eqref{ee3} that 
\[ \sup_{(f,u_0,u)\in\mathcal{I}_T} \left( \inf_{t\in (0,T]} \left( ||u_x(\cdot , t) ||_B - \mathcal{F}_t (f,u_0,u) \right) \right) = 0 ,\]
and hence, the derivative estimate in Proposition \ref{lem233}, according to the above definition, is sharp. To remove such trivial cases, we introduce the following improvement to the above definition; namely, we refer to the derivative estimate in Proposition \ref{lem233} as \textit{non-trivially sharp} on $\bar{D}_T$ when there exists $\alpha >0$ such that 
\[ \sup_{(f,u_0,u)\in\mathcal{I}_T^\alpha} \left( \inf_{t\in (0,T]} \left( ||u_x(\cdot , t) ||_B - \mathcal{F}_t (f,u_0,u) \right) \right) = 0, \]
where 
\[ \mathcal{I}_T^\alpha = \{ (f,u_0,u): (f,u_0,u)\in\mathcal{I}_T \text{ and } ||u_x(\cdot , T)||_B \geq \alpha \}. \]
We can now state the main result in this paper, as

\begin{theorem} \label{tmain}
For any $\alpha , T>0$, there exists $u_0\in \text{BPC}^{\hspace{0.5mm} 1}(\field{R})$ and a sequence $\{(f_n , u_0, u_n)\in\mathcal{I}_T^\alpha \}_{n\in\field{N}}$, such that 
\[ \lim_{ n\to\infty } \left( \inf_{t\in (0,T]} \left( ||u_{nx}(\cdot , t) ||_B - \mathcal{F}_t (f_n,u_0,u_n) \right) \right) = 0 .\]
\end{theorem}

\noindent It then follows immediately from Theorem \ref{tmain} that
\begin{corollary} \label{cmain}
The derivative estimate in Proposition \ref{lem233} is non-trivially sharp on $\bar{D}_T$ for any $T>0$.
\end{corollary}

The paper is structured as follows. In \textsection 2, for $0<p<1$ we introduce \mbox{$( f_p , 0, u^{(p)} )\in \mathcal{I}_T$} with $u^{(p)}:\bar{D}_T\to\field{R}$ specific nontrivial anti-symmetric self similar solutions to [IVP], which correspond to front solutions in \cite{JCMDJN4}. In \textsection 3, we consider a formal limit as $p\to 0$ of a boundary value problem for the ordinary differential equation associated with $u^{(p)}$. In \textsection 4, we establish Theorem \ref{tmain}. Finally in \textsection 5 we discuss alternative approaches to establish Theorem \ref{tmain}, as well as similar problems for which this type of theorem may be established.

\section{The Problem (P$_p$)}
For $p\in (0,1)$, consider [IVP] with nonlinearity $f_p:\field{R}\to\field{R}$ given by \eqref{fp} and initial data $u_0:\field{R}\to\field{R}$ such that $u_0 = 0$. Henceforth we will refer to this as (P$_p$). In \cite[Theorem 3.14]{JCMDJN4} it is demonstrated that, for any $T>0$, there exists a self similar solution $u^{(p)} :\bar{D}_T\to\field{R}$ to (P$_p$) of the form
\begin{equation} \label{S1} u^{(p)}(x,t) = w_p\left(\eta (x,t) \right) t^{1/(1-p)} \ \ \ \forall (x,t)\in \bar{D}_T ,\end{equation}
with $\eta =  xt^{-1/2}$ whilst $w_p:\field{R}\to \field{R}$ is such that $w_p\in C^2(\field{R})$, and 
\begin{equation} \label{28*} w_p'' + \frac{1}{2}\eta w_p' + f_p(w_p) - \frac{1}{(1-p)}w_p = 0 \ \ \ \forall \eta \in \field{R} , \end{equation}
\begin{equation} \label{30*} w(-\eta ) = -w(\eta ) \ \ \ \forall \eta \in\field{R} , \end{equation}
\begin{equation} \label{S2} |w_p(\eta )| < (1-p)^{1/(1-p)} \ \ \ \forall \eta \in \field{R} , \end{equation} 
\begin{equation} \label{S3} w_p(\eta ) \to \pm (1-p)^{1/(1-p)},\   w_p'(\eta )\to 0 \text{ as } \eta \to \pm \infty , \end{equation}
\begin{equation} \label{S4} 0 < w_p'(\eta ) < \sup_{\eta \in \field{R}} |w_p' (\eta )|  = w_p' (0) \ \ \ \forall \eta \in \field{R}\backslash \{ 0 \} , \end{equation}
\begin{equation} \label{29*} w_p'(0) > \frac{(1-p)^{1/(1-p)}}{(1-p)^{1/2}} .  \end{equation}
The function $w_p:\field{R}\to\field{R}$, for $p\in (0,1)$, will be used extensively throughout the rest of the paper. Now, since $u^{(p)}:\bar{D}_T\to\field{R}$ given by \eqref{S1} is a solution to (P$_p$) for any $T>0$, we have that 
\begin{equation} \label{smiles} (f_p , 0 , u^{(p)})\in\mathcal{I}_T. \end{equation}
In addition, it follows from \eqref{fp}, \eqref{S1}, \eqref{S2} and \eqref{S3}, that,
\begin{equation} \label{D1} \mathcal{F}_t(f_p, 0, u^{(p)}) = \frac{(1-p)^{p/(1-p)} \Gamma (1/(1-p))}{\Gamma ( (3-p)/2(1-p))} t^{(1+p)/2(1-p)  } \ \ \ \forall t\in (0,T] , \end{equation}
whilst from \eqref{S1} and \eqref{S4}, 
\begin{equation} \label{D2} ||u_x^{(p)} (\cdot , t ) ||_B = w_p'(0) t^{(1+p)/2(1-p)} \ \ \ \forall t\in (0,T]. \end{equation}
Therefore, via \eqref{D1}, \eqref{D2} and Proposition \ref{lem233}, 
\begin{equation} \label{D3} ||u_x^{(p)}(\cdot , t) ||_B - \mathcal{F}_t (f_p,0,u^{(p)}) = \left( w_p'(0) - \phi (p)\right)t^{(1+p)/2(1-p)} \leq 0 \ \ \ \forall  t\in (0,T] \end{equation}
where $\phi :(0,1)\to\field{R}$ is given by 
\begin{equation} \label{phi} \phi (p) = \frac{(1-p)^{p/(1-p)} \Gamma (1/(1-p))}{\Gamma ((3-p)/2(1-p))} \ \ \ \forall p\in (0,1) .\end{equation}
Furthermore, it follows that the inequality in \eqref{D3} is {\emph{strict}}, by substituting into \eqref{pse} 
\[ |f_p(u^{(p)}(x,t))| < ||f_p(u^{(p)}(\cdot , t))||_B \ \ \ \forall (x,t)\in D_T,\]
which follows from \eqref{S1}, \eqref{S2} and \eqref{S3}, and proceeding with the proof of Proposition \ref{lem233}. We observe that,
\begin{equation} \label{D4} \phi (p) >0 \ \ \ \forall p\in (0,1), \end{equation}
\begin{equation} \label{D5} \phi (p) \to \frac{2}{\sqrt{\pi}} \text{ as }p\to 0^+ .\end{equation}
In addition, it follows from Proposition \ref{DAVEARGHHH1}, with \eqref{fp}, \eqref{S1} and \eqref{S2}, that,
\begin{equation} \label{D6} \mathcal{F}_t(f_p , 0 , u^{(p)}) \leq \frac{2}{\sqrt{\pi}}(1-p)^{p/(1-p)}t^{(1+p)/2(1-p)} \ \ \ \forall t\in (0,T]. \end{equation}
Then, via \eqref{D1} and \eqref{D6}, we have, 
\begin{equation} \label{D7} \phi (p) \leq \frac{2}{\sqrt{\pi}}(1-p)^{p/(1-p)} < \frac{2}{\sqrt{\pi}} \ \ \ \forall p\in (0,1) . \end{equation}
Now, we conclude from the discussion following \eqref{D3} that  
\begin{equation} \label{D8} \inf_{t\in (0,T]} \left( ||u_x^{(p)}(\cdot , t)||_B - \mathcal{F}_t (f_p,0,u^{(p)}) \right) = (w_p'(0) - \phi (p) ) T^{(1+p)/2(1-p)} < 0 . \end{equation}
We also observe from \eqref{S1} and \eqref{29*} that 
\begin{equation} \label{D9} ||u_x^{(p)}(\cdot , T)||_B = w_p'(0)T^{(1+p)/2(1-p)} > \frac{(1-p)^{1/(1-p)}}{(1+p)^{1/2}} T^{(1+p)/2(1-p)} . \end{equation}
A proof of Theorem \ref{tmain} will now follow, up to minor detail, if we are able to construct a sequence $\{p_n\}_{n\in\field{N}}$, such that $p_n\to 0$ as $n\to \infty$, and 
\[ w_{p_n}'(0) \to \frac{2}{\sqrt{\pi}} \text{ as } n\to\infty . \]
It is the construction of such as sequence which we now address. However, before proceeding, it is worth noting from \eqref{D3} and \eqref{29*}, that at this stage, we have, 
\begin{equation} \label{D10} \frac{(1-p)^{1/(1-p)}}{(1+p)^{1/2}} < w_p'(0) < \phi (p) < \frac{2}{\sqrt{\pi}} \ \ \ \forall p\in (0,1). \end{equation} 
We proceed by examining the solution $w_0:\field{R}^+\to\field{R}$ to a boundary value problem in which the ordinary differential equation is a formal limiting form of that in \eqref{28*}, as $p\to 0^+$. We then show that there exists a sequence $\{p_n\}_{n\in\field{N}}$, such that $p_n\to 0$ as $n\to\infty$ and 
\begin{equation} \label{sadface} w_{p_n} \to w_0 \text{ and }w_{p_n}'\to w_0' \text{ uniformly on } [0,X] \text{ as }n\to\infty , \end{equation}
for any $X>0$. The result then follows on observing that $w_0'(0) = 2/\sqrt{\pi}$.

\section{The problem ($S_0$)} \label{sec0}
In this section, we examine the problem given by taking the formal limit as $p\to 0$ in the initial value problem for the ordinary differential equation studied in \cite{JCMDJN4}. We seek a function $w:[0,\infty )\to\field{R}$ such that $w\in C([0,\infty )) \cap C^2((0,\infty ))$ and
\begin{align}
\label{B1} & w'' + \frac{1}{2} \eta w' - w = -1 \ \ \ \forall \eta >0 ,\\
\label{B2} & w(0)= 0,\ \ \ w(\eta ) \to 1 \text{ as }\eta\to\infty ,\\
\label{B3} & w(\eta ) > 0 \ \ \ \forall \eta > 0 .
\end{align}
We refer to this linear inhomogeneous boundary value problem as (S$_0$). We observe that the coefficients in \eqref{B1} are continuous functions of $\eta \in [0,\infty )$. Thus, the homogeneous part \eqref{B1} has two basis functions $w_1,w_2:[0,\infty )\to\field{R}$ and a particular integral $\bar{w}:[0,\infty )\to\field{R}$ after which every solution of \eqref{B1} may be written as,
\[ w(\eta ) = Aw_1(\eta ) + Bw_2(\eta ) + \bar{w}(\eta ) \ \ \ \forall \eta \in [0,\infty ) , \]
with $A,B\in\field{R}$ being arbitrary constants. Inspection, followed by the method of reduction of order allows us to take 
\begin{align*}
& w_1(\eta ) = 2+ \eta^2 \\
& w_2(\eta ) = (2+\eta^2 )\int_\eta^\infty \frac{ e^{-s^2/4}}{(2+s^2)^2} ds \\
& \bar{w}(\eta ) = 1 
\end{align*}
for all $\eta \in [0,\infty )$. It remains to apply conditions \eqref{B2} and \eqref{B3}. These conditions are satisfied if and only if we choose
\[ A=0\text{ and } B=-\frac{4}{\sqrt{\pi}}. \]
Thus (S$_0$) has a unique solution $w=w_0:[0,\infty )\to\field{R}$ given by 
\begin{equation} \label{D11} w_0(\eta ) = 1- \frac{4}{\sqrt{\pi}} (2+\eta^2)I(\eta ) \ \ \ \forall \eta\in [0,\infty ) \end{equation}
where 
\begin{equation} \label{D12} I(\eta ) = \int_\eta^\infty  \frac{ e^{-s^2/4}}{(2+s^2)^2} ds \ \ \ \forall \eta \in [0,\infty ) \end{equation}
and we note that $I(\eta )$ is monotone decreasing in $\eta \in [0,\infty )$ with $I(0) = \sqrt{\pi}/8$ (see \cite[pp. 302, 7.4.11]{MAIAS1}), and $I(\eta )$ decays exponentially as $\eta\to\infty$. Finally, we observe from \eqref{D11} and \eqref{D12} that 
\begin{equation} \label{D13} w_0'(0) = \frac{2}{\sqrt{\pi}} . \end{equation}
In the following section, we proceed to construct the sequence of functions $w_{p_n}:\field{R}\to\field{R}$ for which \eqref{sadface} holds.

\section{Proof of Theorem 1.4}
In this section, we construct  a sequence $\{p_n\}_{n\in\field{N}}$ such that $p_n\in (0,1)$ for all $n\in\field{N}$, $p_n\to 0$ as $n\to\infty$ and $w_{p_n}:\field{R}\to\field{R}$ satisfies
\begin{equation} \label{confusedface} w_{p_n}\to w_0 , \ \ \ w_{p_n}'\to w_0' \text{ uniformly on } [0,X] \end{equation}
for any $X>0$, where $w_0:[0,\infty )\to\field{R}$ is the unique solution to (S$_0$), given by \eqref{D6}. We note that via \eqref{confusedface} and \eqref{D13}, we have 
\[ w_{p_n}'(0) \to \frac{2}{\sqrt{\pi}} \text{ as }n\to\infty ,\]
which is crucial to the proof of Theorem \ref{tmain}.  

Throughout this section we consider $w_p:\field{R}\to\field{R}$ restricted to the domain $[0,\infty )$, so that $w_p=w_p:[0,\infty )\to\field{R}$. To begin, we obtain uniform bounds on $w_p$, $w_p'$ and $w_p''$ for $p\in (0,1)$. We have first,  

\begin{proposition} \label{P1}
Consider $w_p: [0,\infty )\to \field{R}$ with $p\in (0,1)$. Then,
\[ 0\leq w_p(\eta ) < 1 \ \ \ \forall \eta \geq 0 , \]
and
\[ |w_p(\eta_1 ) - w_p(\eta_2 ) | \leq \frac{2}{\sqrt{\pi}} |\eta_1 - \eta_2 | \ \ \ \forall \eta_1 , \eta_2 \geq 0 . \]
\end{proposition}
 
\begin{proof}
It follows from \eqref{S2}, \eqref{S4} and \eqref{D10} that for $p\in (0,1)$
\[ 0 \leq w_p(\eta ) \leq (1-p)^{1/(1-p)} < 1 \ \ \ \forall \eta \geq 0 , \]
and
\begin{equation} \label{1A} 0 < w_p'(\eta ) \leq w_p'(0)< \frac{2}{\sqrt{\pi}} \ \ \ \forall \eta \geq 0. \end{equation}
Therefore, via the mean value theorem with \eqref{1A}, we have
\[ |w_p(\eta_1 ) - w_p(\eta_2 )| \leq \sup_{\theta \in [0,\infty )} \{ w_p'(\theta ) \} |\eta_1 - \eta_2| \leq \frac{2}{\sqrt{\pi}} |\eta_1 - \eta_2 | \ \ \ \forall \eta_1 , \eta_2 \geq 0, \]
as required. 
\end{proof}
\noindent Additionally, we have

\begin{proposition} \label{P2}
Consider $w_p: [0,\infty )\to \field{R}$ with $p\in (0,1)$. Then,
\begin{equation} \label{62*}  0 <  w_p'(\eta ) < \frac{2}{\sqrt{\pi}} \ \ \ \forall \eta \geq 0 \end{equation}
and for any $X>0$,
\[ |w_p'(\eta_1 ) - w_p'(\eta_2 )| \leq \left( \frac{X}{\sqrt{\pi}} + 2 \right) |\eta_1 - \eta_2 |\ \ \ \forall \eta_1 ,\eta_2\in [0,X]. \]
\end{proposition}

\begin{proof}
The first inequality follows from \eqref{S4} and \eqref{D10}. Next, via the mean value theorem,
\begin{equation} \label{2A} |w_p'(\eta_1) - w_p'(\eta_2)| \leq \sup_{\theta \in [0,X] }| w_p''(\theta )| | \eta_1 - \eta_2 | \ \ \ \forall \eta_1 ,\eta_2 \in [0,X]. \end{equation}
Now, from \eqref{28*}, \eqref{S2} and \eqref{62*}, 
\begin{equation} \label{2B} |w_p''(\theta )|  \leq \left| \frac{\theta}{2}w_p'(\theta ) \right| + |w_p(\theta )|^p + \left|\frac{1}{(1-p)}w_p(\theta )\right| \leq \frac{X}{\sqrt{\pi}} + 2(1-p)^{p/(1-p)}   \leq  \frac{X}{\sqrt{\pi}} + 2 \end{equation}
for all $\theta \in [0,X]$. The result then follows from \eqref{2A} and \eqref{2B}.
\end{proof}
\noindent Before we can obtain a result for $u_p''$ corresponding to Proposition \ref{P1} and Proposition \ref{P2}, we need the following.

\begin{proposition} \label{P12}
Consider $w_p:[0,\infty )\to \field{R}$ with $p\in (0,1/2]$. Then
\[ w_p(\eta ) \geq  \begin{cases} \frac{1}{8\sqrt{2}}\eta &, 0\leq \eta \leq \eta ' \\ \frac{\eta '}{8\sqrt{2}} &, \eta > \eta ' \end{cases} \]
where
\begin{equation} \label{12A'} \eta ' =\sqrt{\pi} \left(\sqrt{ 1 + \frac{1}{4\sqrt{2\pi}}} -1 \right) <1 . \end{equation}
\end{proposition}

\begin{proof}
It follows from \eqref{28*} that
\begin{equation} \label{12A} \int_0^\eta w_p''(s) ds = \int_0^\eta \left( -\frac{1}{2}sw_p'(s) + \frac{1}{(1-p)}w_p(s) - (w_p(s))^p \right) ds \ \ \ \forall \eta \in [0,\infty ). \end{equation}
For $p\in (0,1/2]$ we next observe that $H_p:[0,(1-p)^{1/(1-p)}]\to\field{R}$ given by
\begin{equation} \label{12B} H_p(x) = \frac{1}{(1-p)}x - x^p \ \ \ \forall x\in [0,(1-p)^{1/(1-p)}] \end{equation}
satisfies
\begin{equation} \label{12C} H_p(x) \geq -(1-p)^{p/(1-p)} \geq -1 \ \ \ \forall x\in [0,(1-p)^{1/(1-p)}] . \end{equation}
Thus, it follows from Proposition \ref{P2}, \eqref{S2}, \eqref{12A}, \eqref{12B} and \eqref{12C} that 
\begin{equation} \label{12D} w_p'(\eta ) - w_p'(0) \geq \int_0^\eta \left( -\frac{1}{\sqrt{\pi}}s-1\right) ds \geq -\frac{1}{2\sqrt{\pi}}\eta^2 - \eta \ \ \ \forall \eta \in [0,\infty ). \end{equation}
Now, from \eqref{29*} we also have
\begin{equation} \label{12E} w_p'(0) > \frac{(1-p)^{1/(1-p)}}{\sqrt{1+p}} \geq \frac{1}{4\sqrt{2}} \end{equation}
for all $p\in (0,1/2]$. Therefore, it follows from \eqref{12D} and \eqref{12E} that 
\begin{equation} \label{12F} w_p'(\eta ) \geq w_p'(0) - \frac{1}{2\sqrt{\pi}}\eta^2 - \eta \geq \frac{1}{4\sqrt{2}} - \frac{1}{8\sqrt{2}} = \frac{1}{8\sqrt{2}} \ \ \ \forall \eta \in [0, \eta '] \end{equation}
with $\eta '$ given by \eqref{12A'}. Now, it follows from \eqref{12F} and \eqref{30*} that  
\[ w_p(\eta ) = \int_0^\eta w_p'(s)ds \geq \int_0^\eta \frac{1}{8\sqrt{2}} ds = \frac{1}{8\sqrt{2}}\eta \ \ \ \forall \eta \in [0,\eta '] . \]
Finally, via \eqref{S4}, we have
\[ w_p(\eta ) \geq w_p(\eta ') \geq \frac{\eta '}{8\sqrt{2}} \ \ \ \forall \eta \in (\eta ' ,\infty ) ,\]
as required.
\end{proof}
\noindent We now have,  

\begin{remk} \label{R4}
It follows from Proposition \ref{P1} and Proposition \ref{P2} that $\{w_p \}_{p\in (0,1)}$ and $\{w_p'\}_{p\in (0,1)}$ are uniformly bounded and equicontinuous on $[0,X]$ for each $X>0$.
\end{remk}

We next define the sequence $\{ p_n\}_{n\in\field{N}}$ such that $p_n=1/(2n)$ and the sequence of functions $\{ v_n\}_{n\in\field{N}}$ such that  
\begin{equation} \label{u_n} v_n = w_{p_n}:[0,\infty )\to\field{R} , \end{equation} 
after which we may state, 

\begin{lemma} \label{L5}
There exists a function $w_*:[0,\infty )\to\field{R}$ such that $w_*\in C^1([0,\infty ))$ and for any $X>0$, the sequence of functions $\{v_n \}_{n\in\field{N}}$ given by \eqref{u_n} has a subsequence $\{v_{n_l}\}_{l\in\field{N}}$ with ($1\leq n_1 <n_2 < ...$ and $n_l\to\infty$ as $l\to\infty$) that satisfies
\begin{equation} \label{50A} v_{n_l}\to w_* \text{ as }l\to\infty \text{ uniformly on } [0,X]   , \end{equation}
\begin{equation} \label{50B} v_{n_l}'\to w_*'\text{ as }l\to\infty \text{ uniformly on } [0,X]   . \end{equation}
\end{lemma}

\begin{proof}
To begin, fix $X>0$. It follows from Proposition \ref{P1} that the sequence $\{v_n\}_{n\in\field{N}}$ in \eqref{u_n} is uniformly bounded and equicontinuous on $[0,X]$. Therefore, via \cite[Theorem 7.25]{WR1} there exists a subsequence $\{v_{n_j}\}_{j\in\field{N}}$ of $\{ v_{n}\}_{n\in \field{N}}$ ($1\leq n_1 <n_2 < ...$ and $n_j\to\infty$ as $j\to\infty$) and $w_*\in C([0,X])$ such that  
\begin{equation} \label{5A} v_{n_j}\to w_*\text{ as }j\to\infty \text{ uniformly on }[0,X]  , \end{equation}
and
\begin{equation} \label{5A'} w_*(\eta ) = \lim_{j\to\infty }v_{n_j}(\eta ) \ \ \ \forall \eta \in [0,X]. \end{equation}
Moreover, it follows from Proposition \ref{P2} that $\{v_{n_j}'\}_{j\in\field{N}}$ is uniformly bounded and equicontinuous on $[0,X]$. Again, via \cite[Theorem 7.25]{WR1} there exists a subsequence $\{ v_{n_k'}'\}_{k\in\field{N}}$ of $\{ v_{n_j}'\}_{j\in \field{N}}$ and $\tilde{v}\in C([0,X])$ such that
\begin{equation} \label{5B} v_{n_k'}'\to \tilde{v} \text{ as }k\to\infty  \text{ uniformly on }[0,X]  , \end{equation}
and
\begin{equation} \label{5B'} \tilde{v}(\eta ) = \lim_{k\to\infty }v_{n_k'}'(\eta ) \ \ \ \forall \eta \in [0,X]. \end{equation}
It now follows from \eqref{5A}-\eqref{5B'} with \cite[Theorem 7.17]{WR1} that $w_*\in C^1([0,X])$ and 
\begin{equation} \label{5C} w_*'(\eta ) = \tilde{v}(\eta ) \ \ \ \forall \eta \in [0,X]. \end{equation}
Now, since \eqref{5A}-\eqref{5C} holds for any $X>0$, we may define $w_*\in C^1([0,1])$, and then, by taking nested subsequences of the original subsequences, extend the definition of $w_*\in C^1([0,2])$. This process can be repeated so that the definition is extended to $w_*\in C^1([0,\infty ))$, with \eqref{50A} and \eqref{50B} holding on $[0,X]$, with the suitably generalised nested subsequence associated with the interval $[0, [X]+1]$.
\end{proof}

\begin{remk} \label{R6}
Lemma \ref{L5} ensures the existence of a function $w_*:[0,\infty )\to\field{R}$ such that $w_*\in C^1([0,\infty ))$, and there exists a subsequence $\{ v_{n_l}\}_{l\in\field{N}}$ of $\{ v_n\}_{n\in\field{N}}$ given by \eqref{u_n}, that satisfies
\[ w_*(\eta ) = \lim_{l\to\infty} v_{n_l}(\eta ) \ \ \ \forall \eta \geq 0 ,\]
\[ w_*'(\eta ) = \lim_{l\to\infty} v_{n_l}'(\eta ) \ \ \ \forall \eta \geq 0. \]
Additionally, via Proposition \ref{P1}, Proposition \ref{P2} and \eqref{S3}, it follows that
\[ 0\leq w_*(\eta )  \leq 1, \ \ \ 0 \leq w_*' (\eta )  \leq \frac{2}{\sqrt{\pi}} \ \ \ \forall \eta \geq 0 ,\]
and
\[ w_*(0) = 0 , \]
whilst
\[ w_*(\eta ) \geq \begin{cases} \frac{1}{8\sqrt{2}} \eta &, 0 \leq \eta \leq \eta ' \\ \frac{1}{8\sqrt{2}} \eta ' &, \eta > \eta ' \end{cases} \]
via Proposition \ref{P12}.
\end{remk}

\noindent We now have,

\begin{proposition} \label{NNN}
Let $X_2>X_1>0$. Then $w_*\in C^2([X_1,X_2])$ and,
\[ w_*'' + \frac{1}{2}\eta w_*' + 1 - w_* = 0 \ \ \ \forall \eta \in [X_1,X_2]. \]
\end{proposition}

\begin{proof}
Set $X_2>X_1>0$. It then follows from Lemma \ref{L5} that there is a subsequence $\{v_{n_l} \}_{l\in \field{N}}$ of $\{ v_n \}_{n\in\field{N}}$ such that, 
\begin{equation} \label{M1} v_{n_l}\to w_* \text{ as }l\to\infty \text{ uniformly on } [X_1,X_2]   , \end{equation}
\begin{equation} \label{M2} v_{n_l}'\to w_*'\text{ as }l\to\infty \text{ uniformly on } [X_1,X_2]   . \end{equation}
Also, via \eqref{u_n},
\begin{equation} \label{M3} v_{n_l}'' = -\frac{1}{2} \eta v_{n_l}' - (v_{n_l})^{p_{n_l}} + \frac{v_{n_l}}{(1-p_{n_l})} \ \ \ \forall \eta \in [X_1,X_2] .\end{equation}
We now observe that $w_*$ is bounded above zero on $[X_1,X_2]$, via Remark \ref{R6}, and so it follows from \eqref{M1}-\eqref{M3}, that 
\begin{equation} \label{M4} v_{n_l}'' \to -\frac{1}{2} \eta w_*' - 1 + w_* \text{ as }l\to\infty \text{ uniformly on }[X_1,X_2] . \end{equation}
Finally, via \eqref{M4} and \cite[Theorem 7.17]{WR1}, we conclude that $w_*\in C^2([X_1,X_2])$ and 
\begin{equation} \label{M5} v_{n_l}'' \to w_*'' \text{ as } l\to\infty \text{ uniformly on }[X_1,X_2] . \end{equation}
The proof is completed via \eqref{M4}, \eqref{M5} and uniqueness of limits. 
\end{proof}

We now investigate the behavior of $w_*:[0,\infty )\to\field{R}$ as $\eta \to\infty$. To begin, we have,

\begin{lemma} \label{P7}
Consider $w_p:[0,\infty )\to\field{R}$ with $p\in (0,1/2]$. Then
\begin{equation} \label{D14} 0 < w_p'(\eta ) < \frac{2}{\sqrt{\pi}} e^{-\eta^2 / 4} \ \ \ \forall \eta \geq 0, \end{equation}
and 
\[ -2\mathrm{erfc}\left( \frac{1}{2} \eta \right) \leq w_p(\eta ) - (1-p)^{1/(1-p)}\leq 0 \ \ \ \forall \eta\geq 0 . \]
\end{lemma}

\begin{proof}
Via \eqref{28*}, \eqref{S2}, \eqref{S4} and \eqref{29*}, we have 
\begin{equation} \label{7A} w_p'' + \frac{\eta}{2}w_p' = \frac{1}{(1-p)}w_p - f_p(w_p) < 0 \ \ \ \forall \eta \in (0,\infty ). \end{equation}
Therefore,
\begin{equation} \label{7B} w_p'(\eta ) < w_p'(0)  e^{-\eta^2 / 4}\ \ \ \forall  \eta \geq 0 . \end{equation}
The first inequality then follows from \eqref{7B} and Proposition \ref{P2}. An integration of \eqref{D14} then gives
\begin{equation} \label{D15} 0 < w_p(\eta_l ) - w_p(\eta ) < \frac{2}{\sqrt{\pi}} \int_\eta^{\eta_l}  e^{-\lambda^2/4}d\lambda \end{equation}
for any $\eta_l > \eta > 0$. Allowing $\eta_l\to\infty$ in \eqref{D15}, using \eqref{S3}, then results in 
\[ -2\text{erfc}\left( \frac{1}{2} \eta \right) \leq w_p(\eta ) - (1-p)^{1/(1-p)}\leq 0 \ \ \ \forall \eta\geq 0 , \]
as required.
\end{proof}
\noindent We now have

\begin{corollary} \label{P9}
Consider $w_p:[0,\infty )\to\field{R}$ with $p\in (0,1/2]$. Then,
\[ w_p(\eta ) \to (1-p)^{1/(1-p)} \ \text{ as }\ \eta \to\infty \text{ uniformly for } p\in (0,1/2] . \]
\end{corollary}

\begin{proof}
This follows immediately from Lemma \ref{P7}.
\end{proof}

As a consequence of Corollary \ref{P9}, we now have

\begin{lemma} \label{L10}
The function $w_*:[0,\infty )\to \field{R}$ satisfies, 
\[ w_*(\eta ) \to 1 \text{ as }\eta \to \infty . \]
\end{lemma}

\begin{proof}
It follows from Remark \ref{R6} that 
\begin{equation} \label{95} \limsup_{\eta\to\infty} w_*(\eta ) \leq 1 . \end{equation}
Now, from Corollary \ref{P9}, for any $\epsilon >0 $, there exists $\eta_*>0$ (dependent only upon $\epsilon$) such that for all $p\in (0,1/2]$, then  
\begin{equation} \label{10C} w_p(\eta ) \geq (1-p)^{1/(1-p)} - \epsilon \ \ \ \forall \eta\geq \eta^* . \end{equation}
Thus, via \eqref{10C} and Remark \ref{R6}, 
\[ w_* (\eta ) \geq 1-\epsilon \ \ \ \forall \eta \geq \eta^* \]
and so,
\begin{equation} \label{D16} \liminf_{\eta\to\infty} w_*(\eta ) \geq 1 - \epsilon . \end{equation}
Since \eqref{D16} holds for any $\epsilon >0$, then 
\begin{equation} \label{98} \liminf_{\eta\to\infty} w_*(\eta ) \geq 1 . \end{equation}
It follows immediately from \eqref{95} and \eqref{98} that the limit of $w_*(\eta )$ as $\eta \to\infty$ exists and 
\[ \lim_{\eta\to\infty} w_*(\eta ) = 1 , \]
as required.
\end{proof}
\noindent Thus, we have,

\begin{remk} \label{R11}
Via Remark \ref{R6} and Lemma \ref{L10}, the function $w_*:[0,\infty )\to \field{R}$ satisfies; 
\[ w_*\in C^1([0,\infty ))\cap C^2((0,\infty )) ,\]
\[ w_*(\eta ) \leq 1 \ \ \ \forall \eta \in [0,\infty ) ,\]
\[ w_*(\eta ) \geq \begin{cases} \frac{1}{8\sqrt{2}} \eta &, 0 \leq \eta \leq \eta ' \\ \frac{1}{8\sqrt{2}}\eta ' &, \eta > \eta ' ,\end{cases} \]
\[ 0 \leq w_*'(\eta ) \leq \frac{2}{\sqrt{\pi}}\ \ \ \forall \eta \in [0,\infty ), \]
\[ w_*(0)= 0,\ \ \ \lim_{\eta\to\infty} w_*(\eta ) = 1,\]
with 
\[ \eta ' = \sqrt{\pi} \left(\sqrt{ 1 + \frac{1}{4\sqrt{2\pi}}} -1 \right) <1 . \]
\end{remk}

We now have

\begin{proposition} \label{T14}
The function $w_*:[0,\infty )\to\field{R}$ is given by
\[ w_*(\eta ) = 1 - \frac{ 4(2+\eta^2 ) I(\eta )}{\sqrt{\pi}} \ \ \ \forall \eta \geq 0. \]
\end{proposition}

\begin{proof}
It follows from Proposition \ref{NNN}, that 
\[ w_*''+ \frac{1}{2}\eta w_*' + 1 - w_* = 0 \ \ \ \forall \eta >0 . \]
We conclude, via Remark \ref{R11}, that $w_*:[0,\infty )\to\field{R}$ satisfies the boundary value problem (S$_0$). It has been established in \textsection \ref{sec0} that (S$_0$) has a unique solution given by \eqref{D11}-\eqref{D12}, as required. 
\end{proof}

We immediately have,
\begin{corollary} \label{C15}
There exists a subsequence $\{p_{n_l}\}_{l\in\field{N}}$ of $\{p_n\}_{n\in\field{N}}$ such that 
\[ w_{p_{n_l}}'(0)\to \frac{2}{\sqrt{\pi}} \text{ as }l\to\infty . \]
\end{corollary}

\begin{proof}
It follows directly from \eqref{u_n} and Remark \ref{R6} that there exists a subsequence $\{ p_{n_l}\}_{l\in\field{N}}$ of $\{ p_n\}_{n\in\field{N}}$ such that,
\[ w_{p_{n_l}}'(0) \to w_*'(0) \text{ as }l\to\infty . \]
However, from Proposition \ref{T14}, \eqref{D11} and \eqref{D13},
\[ w_*'(0) = w_0'(0) = \frac{2}{\sqrt{\pi}} \]
and the proof is complete. 
\end{proof}

We are now able to give the proof of our main result

\begin{proof}[Proof of Theorem \ref{tmain}]
First fix $\alpha >0$ and $T>0$. Next consider the subsequence $\{ p_{n_l}\}_{l\in\field{N}}$ of $\{ p_n\}_{n\in\field{N}}$ corresponding to that in Corollary \ref{C15}. Let the constant 
\begin{equation} \label{D19} c(\alpha , T) = \frac{\sqrt{\pi}}{2T^{1/2}} (\alpha + 1) . \end{equation}
We now introduce the sequence of functions $\{ u^{(l)}:\bar{D}_T\to\field{R} \}_{l\in\field{N}}$ as 
\begin{equation} \label{D20} u^{(l)}(x,t)= \begin{cases} w_{p_{n_l}}\left(\frac{x}{\sqrt{t}}\right)(c(\alpha , T) t)^{1/(1-p_{n_l})} &; (x,t)\in [0,\infty )\times (0,T] \\ -w_{p_{n_l}}\left(\frac{-x}{\sqrt{t}}\right)(c(\alpha , T) t)^{1/(1-p_{n_l})} &; (x,t)\in (-\infty , 0) \times (0,T] \\ 0 &; (x,t)\in (-\infty , \infty )\times \{ 0 \} .\end{cases} \end{equation}
It is straightforward to verify directly, via \eqref{D19}, \eqref{D20}, \eqref{u_n}, \eqref{S1} and \eqref{smiles}, that for each $l\in\field{N}$, 
\begin{equation} \label{D21} (c(\alpha , T)f_{p_{n_l}} , 0 ,u^{(l)}) \in\mathcal{I}_T .\end{equation}
In addition, via \eqref{D20}, \eqref{S1} and \eqref{S4}, we have
\begin{equation} \label{D22} ||u_x^{(l)}(\cdot , t)||_B = w_{p_{n_l}}'(0) c(\alpha , T)^{1/(1-p_{n_l})}t^{(1+p_{n_l})/2(1-p_{n_l})} \ \ \ \forall t\in [0,T], \end{equation}
and so, in particular,
\begin{equation} \label{D23} ||u_x^{(l)}(\cdot , T)||_B \to (\alpha + 1) \text{ as } l\to\infty , \end{equation}
via \eqref{D22} and \eqref{D19} with Corollary \ref{C15}. Therefore, there exists $L\in \field{N}$ such that 
\begin{equation} \label{D24} || u_x^{(l)}(\cdot ,T)||_B \geq \alpha \ \ \ \forall l\geq L. \end{equation}
Therefore, for each $l\geq L$, 
\begin{equation} \label{D25} (c(\alpha , T)f_{p_{n_l}} , 0 ,  u^{(l)})\in\mathcal{I}_T^\alpha . \end{equation}
Finally, it follows as in \eqref{D8}, that for each $l\geq L$,
\begin{align*}
\inf_{t\in [0,T]} & \left( ||u_x^{(l)}(\cdot , t)||_B - \mathcal{F}_t (c(\alpha , T) f_{p_{n_l}} , 0 , u^{(l)})\right) \\
& = c(\alpha , T)^{1/(1-p_{n_l})} T^{(1+p_{n_l})/2(1-p_{n_l})}\left( w_{p_{n_l}}'(0) - \phi (p_{n_l}) \right) 
\end{align*}
and so, via Corollary \ref{C15} and \eqref{D5},
\begin{align*}
\lim_{l\to\infty} & \left( \inf_{t\in [0,T]}  \left(  ||u_x^{(l)}(\cdot , t)||_B - \mathcal{F}_t (c(\alpha , T) f_{p_{n_l}} , 0 , u^{(l)}) \right) \right) \\
& = c(\alpha , T)T^{1/2}\left( \frac{2}{\sqrt{\pi}} - \frac{2}{\sqrt{\pi}} \right) = 0 ,
\end{align*}
and the proof of Theorem \ref{tmain} is complete.
\end{proof}

\section{Discussion}
We note here that it is not possible to establish a proof of Theorem \ref{tmain} with a sequence of the form $\{ (g_n,u_0,u^{(n)})\in \mathcal{I}_T\}_{n\in\field{N}}$ with $g_n:\field{R}\to\field{R}$ anti-symmetric, Lipschitz continuous, and such that $g_n(u)\to 1$ as $n\to\infty$ for each $u>0$, with $u^{(n)}:\bar{D}_T\to\field{R}$ the unique solution to 
\[ u_t^{(n)}-u_{xx}^{(n)}-g_n(u^{(n)})=0 \ \ \ \text{ on } D_T \]
\[ u^{(n)}= 0 \ \ \ \text{ on }\partial D .\]
This follows since $u^{(n)} = 0$ on $\bar{D}_T$ for each $n\in\field{N}$, via uniqueness of solutions (see \cite[Theorem 6.1]{JMDN1}). However, we anticipate that a proof of Theorem \ref{tmain} can be established, somewhat more generically, by considering a sequence of the form $\{ (g_n,u_0,u^{(n)})\in \mathcal{I}_T\}_{n\in\field{N}}$, with $g_n$ and $u^{(n)}$ defined as above, but instead with non-zero initial data $u_0:\field{R}\to\field{R}$ given by
\[ u_0(x) =  w_0(x/\lambda^{1/2})\lambda \ \ \ \forall x\in\field{R},\]
for some fixed $\lambda >0$, with $w_0$ given by \eqref{D11}-\eqref{D12}. 

We anticipate that the approach adopted here can be used to show a similar sharpness result for the natural functional derivative estimate associated with solutions $u:\field{R}^N\times [0,T]\to\field{R}$ to the Cauchy problem,
\begin{equation} \label{final1} u_t - \Delta u = f(u) \ \ \ \text{ on } \field{R}^N\times (0,T], \end{equation}
\begin{equation} \label{final2} u=u_0 \ \ \ \text{ on } \field{R}^N\times \{ 0 \} ,\end{equation}
\begin{equation} \label{final3} u\in C(\bar{D}_T)\cap C^{2,1}(D_T)\cap L^\infty (\bar{D}_T).\end{equation}
Additionally, we note that it is likely that results of similar type to Theorem \ref{tmain} can be established for functional derivative estimates of the Dirichlet and Neumann problems associated with \eqref{final1}-\eqref{final3}.

\section*{Acknowledmnets} The first author was financially supported by an EPSRC PhD Prize. Additionally, the authors thank Prof. J. Bennett for reading an early version of the manuscript, and for his helpful suggestions.


\begin{thebibliography}{ww}
\bibitem{MAIAS1} M. Abramowitz and I. A. Stegun. \textit{Handbook of mathematical functions - with formulas, graphs, and mathematical tables.} (New York: Dover, 1970)
\bibitem{JMDN1} J. C. Meyer and D. J. Needham. \textit{The Cauchy Problem for Non-Lipschitz Semi-Linear Parabolic Partial Differential Equations.} London Mathematical Society Lecture Notes, Vol 419. (Cambridge: Cambridge University Press, 2015)
\bibitem{JMDN3} J. C. Meyer and D. J. Needham. Aspects of Hadamard well-posedness for classes of non-Lipschitz semi-linear parabolic partial differential equations. To appear in \textit{Proc. R. Soc. Edinburgh Sect. A} 2016  
\bibitem{JCMDJN4} J. C. Meyer and D. J. Needham. The evolution to localized and front solutions in a non-Lipschitz reaction-diffusion Cauchy problem with trivial initial data. Submitted to \textit{J. Differential. Equations} 
\bibitem{WR1} W. Rudin. \textit{Principles of Mathematical Analysis (3$^{rd}$ Ed.)} (Singapore: McGraw-Hill, 1976) 
\end{thebibliography}
\end{document}